\documentclass[12pt]{amsart}

\usepackage{amssymb}
\usepackage{enumerate}

\usepackage{graphicx}

\usepackage{mathrsfs}
\usepackage{amsmath}
\usepackage{hyperref}

\makeatletter
\@namedef{subjclassname@2010}{%
  \textup{2010} Mathematics Subject Classification}
\makeatother

\newtheorem{thm}{Theorem}[section]
\newtheorem{coro}[thm]{Corollary}
\newtheorem{lem}[thm]{Lemma}
\newtheorem{prop}[thm]{Proposition}

\theoremstyle{definition}
\newtheorem{defn}[thm]{Definition}
\newtheorem{rem}[thm]{Remark}
\newtheorem{ex}[thm]{Example}
 \newtheorem{notat}[thm]{Notation}

\numberwithin{equation}{section}

\newcommand{\mC}{\mbox{${{\mathcal C}}$}}
\newcommand{\mR}{\mbox{${{\mathcal R}}$}}

\newcommand{\C}{\mbox{${{\mathbb C}}$}}
\newcommand{\G}{\mbox{${{\Gamma_2}}$}}

\newcommand\atopn[2]{\genfrac{}{}{0pt}{}{#1}{#2}}

\begin{document}

\title[On Cannon cone types and multiplicative
functions]{On Cannon cone types and 
vector-valued multiplicative functions for genus-two-surface-group}
\author[S. Saliani]{Sandra Saliani}
\address{Dipartimento di Matematica, Informatica ed Economia\\
Universit\`{a} degli Studi
 della Basilicata\\ Viale dell'ateneo lucano 10, 85100 Po\-ten\-za, ITALIA}
\email{sandra.saliani@unibas.it}

\date{\today}

\begin{abstract}
We consider Cannon cone types for a surface group of genus $g$, and we give algebraic criteria for
establishing the cone type of a given cone and of all its sub-cones. We also re-prove
that the number of cone types is exactly $8g(2g - 1)+1.$
In the genus $2$ case, we explicitly provide the $48\times 48$ matrix of cone types, $M,$
and we prove that $M$ is  primitive, hence Perron-Frobenius.
Finally we define  vector-valued multiplicative functions and we show how to compute
their values by means of $M$.
\end{abstract}



\maketitle

\section{Introduction}

Let $\Gamma_g$ be a surface group of genus $g$.
There are several definitions of cone types available for surface groups, in this
paper we are interested on those called simply
{\it Cannon cone types} \cite{Ca}, not to be confused to the
{\it canonical Cannon cone types} \cite{FP}.

The aim of this paper is to  give an overview on Cannon cone
types, provide a matrix  (for $g=2$), called matrix of cone types, whose columns
are cone types of successors of any  element of $\Gamma_2$ with a given cone type,
show that it is a Perron-Frobenius matrix, and to apply it in the computation
of elementary multiplicative functions.

Even though some results on cone types are ``folklore'' by
now, as far as we know there are no results available in the literature that
can help us in  constructing  such a matrix. We try to fill this gap in the
present work. Our approach to cone types is more of combinatorial/algebraic type
than geometric.

We  believe that there is a  necessity to explicitly provide the matrix of cone types,
for its potential use in numerical algorithms; this feeling, at the time of Cannon's
manuscript,  maybe was not so urgent.
Recently, instead, in the work of Gouezel \cite{G}, the matrix
of  {\it canonical Cannon cone types} has been used to estimate numerically
 the lower bound of the spectral radius of a random walk  on a
genus $2$ surface group, improving a previous result by Bartholdi \cite{B}.

Another  possible application of the matrix of cone types is related
to the construction of vector-valued elementary multiplicative functions,
and the associated new class of representations on surface groups.
The latter have been already defined for free
groups by
Kuhn and Steger in \cite{KS} (see also  \cite{KSS}), further extended
to virtual free groups by Iozzi, et al. in \cite{IKS}, and presently object of
 a work in progress on surface groups by Kuhn, Steger and their collaborators.

The structure of the paper is the following: after a review on cone types in
Section \ref{cannon}, in Section \ref{how} we establish an algebraic criterion to determine any
element's cone type (which is equivalent to determine the cone type of each cone);
this will lead us also to a proof for the well known
fact that there are exactly $8g(2g - 1)+1$ cone types in $\Gamma_g$.
In Section \ref{successors} we  provide cone types for each successors of the $48+1$
possible cone types in $\Gamma_2$. In Section \ref{matrix} we construct the matrix of
cone types and we  show that it is a Perron-Frobenius matrix. Finally, in Section \ref{multi}
we apply the matrix of cone types in the computation of elementary multiplicative functions.

In the meanwhile, we provide a drawing of (part of) the Cayley graph of the genus $2$ surface group
(octagons group)  in a form that we believe new and hopefully useful,
 so to facilitate the reader to imagine a figure that is not easy to describe in a drawing.

\section{Cannon cone types}\label{cannon}
To fix notations we recall some basic concepts on hyperbolic groups in the sense
of Gromov, such as
distance, length, geodesic, Cayley graph,  etc...
referring to  Ohshika's book \cite{Oh} for more details.

\begin{defn}
Given a group $G,$ a subset $A\subset G$ is called a generator
system of $G$ if every element of $G$ is expressed as a product of
elements of $A$. $G$ is said finitely generated if it has
a finite generator system.
\end{defn}

In this paper we assume that a generator system of $G$ is symmetric, i.e. closed under inverses.

If $G$ has a generator system $A$, then there is  a canonical
surjective group homomorphism $p:F(A)\rightarrow G$,
whose kernel  is called the set of relators.
Here $F(A)$ is the free group on $A$,
identified with the set of  reduced words on $A$, i.e., words
in which an element and its inverse are not juxtaposed. We consider the identity as the empty word.

If $R\subset G$ is a subset, then we denote by $<<R>>$ the normal closure of
$R$ in $G$, which is the intersection of all normal subgroups containing $R$.
Intuitively, this is the smallest normal subgroup containing $R$.
It is easy to see that the elements of $<<R>>$ are
$$g_1x^{n_1}_1 g^{-1}_1 g_2x^{n_2}_2 g^{-1}_2 \dots g_k x^{n_k}_k g^{-1}_k,$$
for $n_1,\dots, n_k\in\mathbb{Z},$ $x_1\dots x_k\in R$, and $g_1,\dots, g_k\in G$ (not
necessarily distinct).

\begin{defn}
If  $G$ has a finite generator system $A,$ we say that G is finitely presented
if there is a finite set $R=\{w_1, \dots, w_n\}\subset \ker{p}\subset F(A)$
such that $<<R>>=\ker{p}.$ Hence $G \cong F(A)/<<R>>$. In that case we write
$G=\langle A|R\rangle$ and we call such a presentation of $G$ a finite presentation.
The words $w_1,..., w_n$ are called relators.
\end{defn}

The fundamental group, $\Gamma_g,$ of a compact
surface of genus $g \geq 2$, is a finitely presented group. Its usual
presentation is
$$
\Gamma_g =\langle a_1,\dots,a_g,b_1,\dots,b_g | [a_1,b_1]\cdot\cdot\cdot
[a_g, b_g] \rangle.
$$
where the bracket means the usual commutator $[a,b]=aba^{-1}b^{-1}.$

In this paper we shall deal manly with $g=2,$ and in this case, for simplicity,
we shall write and  fix the set of generators as follows
\begin{equation}\label{gamma2}
\Gamma_2 =\langle a,b,c,d | [a,b][c, d] =aba^{-1}b^{-1}cdc^{-1}d^{-1}\rangle.
\end{equation}

\begin{defn}
Let $G$ be a discrete, finitely generated group with a finite generator system
$A$.  The Cayley graph $\mathcal{G}$ of $G$ with respect to
$A$ is a graph  defined as follows.
\begin{enumerate}
\item The vertices of $\mathcal{G}$  are the elements of $G$.
\item The (unoriented) edges are (non-ordered) couples $(x,xa)$ with $a \in A$.
\end{enumerate}
\end{defn}

We can introduce a metric, denoted by $d$, on a Cayley graph
by letting the length of every edge be $1$ and defining the distance
between two vertices to be the minimum length of edges joining them.

The metric on $\mathcal{G}$  induces a metric
on $G$ when the latter is identified
with the set of vertices of $\mathcal{G}$. We call this metric on $G$
(still denoted by $d$) the word
metric with respect to $A$. In particular, for $x\in G$, we call the distance from
the identity ``e'', with respect to the word metric, the length of $x$, and we
denote it by $|x|$.

We have also that, for all $x,y,z\in G$,  $d(xy,xz)=d(y,z).$

\begin{defn}
A geodesic segment joining two vertices $x ,y$ in $\mathcal{G}$
(or, more briefly, a geodesic from $x$ to $y$) is a map $f$ from a closed interval $[0, l]\subset\mathbb{R}$
to $\mathcal{G}$ such that $f(0) = x$, $f(l) = y$ and $d(f(t), f(s)) = |t- s|$ for all $t, s \in [0, l]$ (in
particular $l= d(x, y)$. When there is
no  confusion, we also call the image of $f$
a geodesic segment with endpoints $x$ and $y$ and we denote it by $\overline{xy}$. We should note that such a geodesic
segment need not be unique.

If in the geodesic segment $\overline{xw}$ we have $x=e$ and
$w=y\in G$, we say that $w$ is a geodesic word (representing $y$). In this case
$d(e,y)=|w|$.
\end{defn}

\begin{defn} Given three points $x,y,z\in \mathcal{G}$, a  geodesic
triangle $\Delta_{x,y,z}$ on $\mathcal{G}$ with vertices $x, y, z,$ is formed by
three  geodesic  segments $\overline{xy},\overline{yz},\overline{zx}.$
 The group $G$ is said hyperbolic if there exists a constant $\delta > 0,$
depending only on $G$, such that, for any geodesic triangle $\Delta_{x,y,z},$ one has that each $u\in\overline{xy}$ is at distance at
most $\delta$ from $\overline{zx}\cup\overline{yz}.$
\end{defn}
One should note that the constant $\delta$ in the previous definition has no much
importance except in the case $\delta=0$ ($\mathbb{R}$-trees).

\begin{rem}
 $\Gamma_g$ is hyperbolic.
\end{rem}

Its Cayley graph, $\mathcal{G}$, is a planar graph, a tessellation of the hyperbolic space $\mathbb{H}^2$ with the following properties:
\begin{enumerate}
\item Every vertex belongs to $4g$ polygons each of $4g$ edges and
vertex angle $\frac{2\pi}{4g};$
\item Every two polygons share one (and only one) edge;
\item $\mathcal{G}$ is bipartite and self-dual.
\end{enumerate}

\begin{defn}
Given any two vertices $ x, y \in\mathcal{G}$ we say that $y$ is a successor
of $x$ if $(x, y)$ is an edge and $|y| = |x|+ 1.$ In this case $x$ is called a predecessor of $y$.
\end{defn}

A simple realization for (part of) the Cayley graph of $\Gamma_2$ is given in
Figure \ref{fig:ottagono1}, available also at
\url{https://www.geogebra.org/m/Jqayn5UZ}

Its center vertex is the identity ``e'', and, given any vertex, any  successor is
obtained by
juxtaposing generators in counterclockwise verse in this order
\begin{equation}\label{ordine}
a,\quad d,\quad c^{-1},\quad d^{-1},\quad c,\quad b,\quad a^{-1},\quad b^{-1}.
\end{equation}
\begin{figure}[ht]
	\includegraphics[scale=0.78]{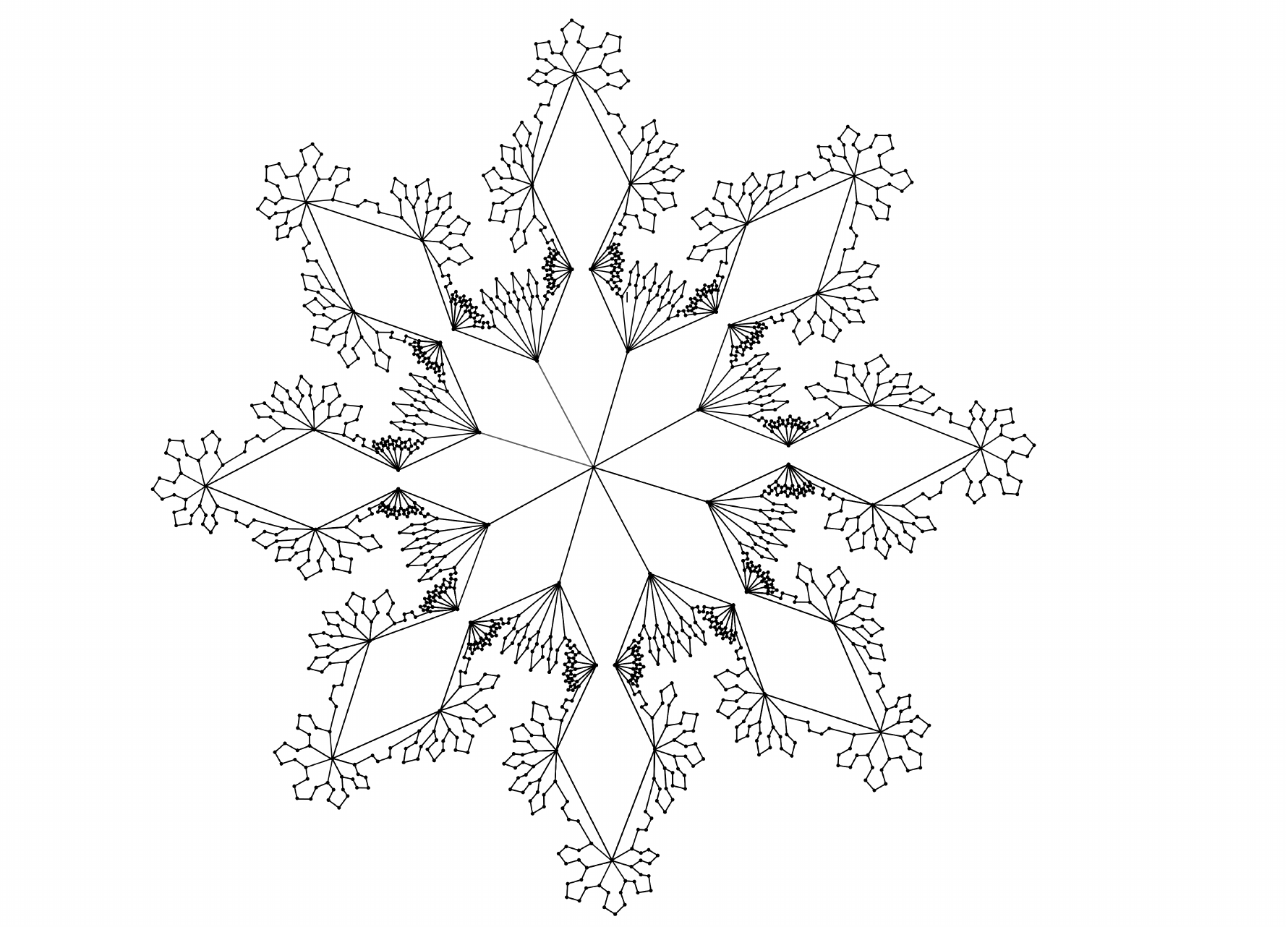}
	\caption{The Cayley graph of $ \Gamma_2$ (part of). }
	\label{fig:ottagono1}
\end{figure}

We now give the definition of cone type for elements of $G$.

\begin{defn}[\cite{BH}]\label{bridson} Let $G$ be a group with finite generating
set $A$ and corresponding word metric $d$. 		
											
The cone type of an element $x\in G,$ denoted by $\mC(x)$, is the set of words $z\in F(A)$ such that
$$d(e, x z) = d(e, x) + |z|,$$
(hence $d(e, x z) = d(e, x) + |z|=d(e,x)+d(x,xz)$). 		
							
In other words, if $x$ is represented by a geodesic word $u$, then the cone type of
$x$ is the set of words $z$ such that $uz$ is also a geodesic.	
\end{defn}

An alternative definition of cone type  involves the cone of
a vertex.
\begin{defn}[\cite{G}]
The cone of a vertex  $x\in \mathcal{G}$, ${\mC}(e,x),$
is the set of vertices  $y\in \mathcal{G}$ for which there is a
geodesic from $e$ to $y$ going through $x$.	
	
The cone type of $x$ is then defined as the set
	 $$\{x^{-1} y, \;\text{for}\; y \in {\mC}(e,x)\}=x^{-1}{\mC}(e,x).$$									
\end{defn}

We see that the two definitions coincide once we identify $G$ as the vertex set
of $\mathcal{G}$, since if $z\in F(A)$ and $d(e, x z) = d(e, x) + |z|$, then
$z=x^{-1}xz$ and $xz\in {\mC}(e,x),$ since $|z|=d(x,xz).$
On the other hand, if $y\in{\mC}(e,x),$ then
$$d(e,x x^{-1}y)=d(e,y)=d(e,x)+d(x,y)=d(e,x)+d(e, x^{-1}y).$$

Next we consider the definition of cone type for a cone, which relies on
the action by isometries of the group $G$ on its Cayley graph, this action is
simply transitive on the vertices.
\begin{defn}\label{tk}
Given two vertices  $x, y \in \mathcal{G}$ the cone at vertex $y$ is
$$\mC(x, y) = \{z \in G ,\, d(x, z) = d(x, y) + d(y, z)\}.$$

The group $G$ acts on the the collection of cones by (left) translation
$$ z\mC(x, y)=\mC(zx, zy),\quad z\in G.$$

We say that two cones have the same type if they are in the same orbit.
We may as well identify the
set of cone-types with the set of cones $\mC(x, e)$ whose vertex is
$e \in G$.
\end{defn}

From
\begin{eqnarray*}
\mC(x)& =& \{z \in G ,\; d(e, xz) = d(e, x) + d(x, xz)\}=x^{-1}{\mC}(e,x)\\
&  =&
\{z \in G ,\; d(x^{-1},z) = d(x^{-1},e) + d(e,z)\}\\
&= &\mC(x^{-1}, e),
\end{eqnarray*}
 we see that  the cone type of $x$, $\mC(x)$, is a representative of
the cone type of $\mC(x^{-1}, e)$, and the latter cone has the same cone type of
${\mC}(e,x)$ (same orbit).

A free group of rank $m$ has $2m + 1$ cone types with respect to any
set of free generators. For more general hyperbolic groups we can refer to
a result due to Cannon:
\begin{thm}[Theorem 2.18, \cite{BH}]\label{can}
If a group $G$ is hyperbolic, then it has only finitely many cone types
(with respect to any finite generating set).	
\end{thm}	
\begin{proof}
 The proof is based on the following result:

Let $r\geq 1$ be an integer.
Define the $r$-level of $g \in G $ as the set of elements $h$ satisfying the
following
$$
|h|\leq r \;\text{and}\; |gh| < |g|.
$$

If the Cayley graph of $G$ is $\delta$-hyperbolic,
the constant $r=2\delta+3$  is
such that
if two elements $g_1$ and $g_2$ have the same $r$-level, then the two
cone types of  $g_1$ and $g_2$ are the same.
\end{proof}

\section{An algebraic criterion}\label{how}

In order to determine all the possible cone types in $\Gamma_g$, we shall
establish an algebraic criterion to determine any element's cone type
(which is equivalent, as shown in the previous section, to determine the cone type
of each cone);
this will lead us also to a proof for the
well known fact that there are exactly $8g(2g-1)$ cone types in $\Gamma_g,$
besides the cone type of the identity element.
As far as we know, no proof of this fact is available in the literature.

It is important for us to determine the exact cone types of an element (and
  of each successor), in order to encode this informations in
 a matrix  useful for  algorithmic computations.

We focus on $\Gamma_2,$ the general case being similar.

We need first some Lemmas.
\begin{lem}\label{lemma1}

 If $w_1 w_2 \dots w_n$ is a geodesic word where each $w_i$ is a generator, then
$${\mC} (w_1 w_2 \dots w_n)\subset{\mC} (w_2 \dots w_{n})\subset \dots
\subset {\mC}(w_n).$$

\end{lem}
\begin{proof} It is sufficient to show the first inclusion.

If $z\notin {\mC} (w_2 \dots w_{n})$, since $w_2 \dots w_{n}$ is a geodesic word,
$$d(e, w_2 \dots w_{n} z) < d(e, w_2 \dots w_{n}) + |z| = n-1 + |z|$$
implies
\begin{eqnarray*}
d(e, w_1 w_2 \dots w_{n} z) &\leq &d(e,w_1) + d(w_1, w_1 w_2 \dots w_{n} z) \\
&=& d(e,w_1) + d(e,  w_2 \dots w_{n} z) \\
&< & 1+n-1 + |z|
=d(e, w_1 w_2 \dots w_{n}) + |z|.
\end{eqnarray*}

Hence $z\notin {\mC}(w_1 w_2 \dots w_{n}).$
\end{proof}

If we look for a proof of the opposite inclusion, we need to consider either any cyclic
permutation of the relator
$$[a,b][c,d]=aba^{-1}b^{-1}cdc^{-1}d^{-1},$$
or of its inverse.
To be short, we say that an element of $\Gamma_2$  belongs to $\mR$ if it is
 represented by a sub-word of a cyclic permutation of the relator or of its inverse, or both
(for example $ba^{-1}b^{-1}c\in \mR$ while $ba^{-1}b^{-1}a\notin \mR$).
Note that  elements in $\mR$ have length at most $4$.
Also, from now on, when considering a geodesic word, say $u_1 u_2\dots u_n$, we
intend that each $u_i$ is a generator.

\begin{lem}\label{lemmatutto}
Let $u_1,\dots,u_i$ be one, two or three generators,  so that $i=1,2,3.$
 Let $y\in \Gamma_2$, such that $u_1\dots u_i\in{\mathcal C}(y).$
Assume that
for {\bf any geodesic word} ${ w_1\dots w_J},$  such that $y=w_1\dots w_J$, the (geodesic) word
 $w_J u_1\dots u_i$ does not belong to $\mR.$
Then:
\begin{enumerate}
\item If $z$  is such that $u_1\dots u_i z$ is a geodesic word, then
for any geodesic $ w_1\dots w_J$ such that $y = w_1\dots w_J,$ the word\\
$w_1\dots w_J u_1\dots u_i z$ is  geodesic, too.
\item ${\mC}(u_1\dots u_i)\subset{\mC}(y u_1\dots u_i),$  and so
${\mC}(u_1\dots u_i)={\mC}(y u_1\dots u_i).$
\end{enumerate}
\end{lem}
\begin{proof} The second sentence follows from the first one and Lemma \ref{lemma1}.
Assume, on the contrary, that there exists a geodesic word $w_1\dots w_J=y$
such that $w_1\dots w_J u_1\dots u_i z$ is not a geodesic word. Then it contains
either couples like $ss^{-1}$ ($s$ being a generator)  or a sequence of (at least)
 $5$ generators in a
cyclic permutation of the relator $[a,b][c,d]$ (or its inverse), or both.

The first case is excluded, since both $w_1\dots w_J u_1\dots u_i$ and
$u_1\dots u_i z$ are geodesic words.

In the second case, since both $w_1\dots w_J u_1\dots u_i$ and $u_1\dots u_i z$
are geodesic words, and $i\leq 3,$
 we get that
the sequence of $5$  must contain $w_J u_1\dots u_i$ against our assumption.

\end{proof}
\begin{ex}
Applying the Lemma \ref{lemmatutto} to points such as $x=bc,$ we get
${\mC} (c)= {\mC}(bc).$

For points such as $x=aba,$ it yields
${\mC} (aba)= {\mC}(ba).$ Note that $ba\in\mR$
and so the procedure stop here.

Consider now points such as $x=abcd.$  In this case
neither $bcd$ nor $abcd$ belong to $\mR$
so we get
$${\mC}(abcd)={\mC}(bcd)={\mC}(cd).$$

The above lemma does not apply to $x=dcd^{-1}c^{-1}a^{-1}dc;$ in this case note
that $c^{-1}a^{-1}dc\notin\mR$
but since
$$x=dcd^{-1}c^{-1}a^{-1}dc = aba^{-1}b^{-1}a^{-1}dc,$$
and $b^{-1}a^{-1}dc\in\mR,$
the cone type
of $x$ is not the same of $a^{-1}dc.$
\end{ex}
Finally we have
\begin{lem}\label{lemma4}
Let $u_1, u_2, u_3,u_4$ be four generators, such that $u_1 u_2 u_3 u_4$ is a
geodesic word in $\mR.$

Let $y\in \Gamma_2$, such that $u_1 u_2 u_3 u_4\in{\mathcal C}(y).$
We have:
\begin{enumerate}
\item If $z$  is such that $u_1 u_2 u_3 u_4 z$ is a geodesic word,
then for any geodesic word $ w_1\dots w_J$ such that $y=w_1\dots w_J,$ we get
that $w_1\dots w_J u_1 u_2 u_3 u_4 z$ is a geodesic word, too.
\item 
 ${\mC}(u_1 u_2 u_3 u_4) = {\mC}(yu_1 u_2 u_3 u_4).$
\end{enumerate}
\end{lem}
\begin{proof}
The second sentence follows from the first one and Lemma \ref{lemma1}.
Assume, on the contrary, that there exists a geodesic word $w_1\dots w_J=y$
such that $w_1\dots w_J u_1 u_2 u_3 u_4 z$ is not a geodesic word, then it contains
either couples like $s s^{-1}$ ($s$ being a generator)  or a sequence of (at least)
 $5$ generators in a
cyclic permutation of the relator $[a,b][c,d]$ (or its inverse), or both.
\\
The first case is excluded, since both $w_1\dots w_J u_1 u_2 u_3 u_4$ and
$u_1 u_2 u_3 u_4 z$ are geodesic words.

In the second case, since  $u_1 u_2 u_3 u_4 z$ is a geodesic words,
 and since $u_1 u_2 u_3 u_4\in\mR,$
we get that
the sequence of $5$  must contain $w_J u_1 u_2 u_3 u_4$ against the fact that $w_1\dots w_J u_1 u_2 u_3 u_4$ is a
geodesic word.
\end{proof}

\begin{ex}
Consider again
$x=dcd^{-1}c^{-1}a^{-1}dc = aba^{-1}b^{-1}a^{-1}dc.$
Since $b^{-1}a^{-1}dc\in\mR,$ the cone type
of $x$ is  the same of $b^{-1}a^{-1}dc.$
\end{ex}

The following lemma states the uniqueness of cone types for elements in $\mR$.
\begin{lem}\label{uni}
Elements in $\mR$ (of length at most $4$)  have distinct cone types.

Hence, if $x,z\in\mR$ and $\mathcal{C}(x)=\mathcal{C}(y),$ then $x=y$.
\end{lem}
\begin{proof} For any couple of words $x,z\in\mR$
it is easy to explicitly provide an element $y\in{\mathcal C}(x)$ not in
$\mathcal{C}(z)$ and vice versa. For example $b^{-1}a^{-1}d\in {\mathcal C}(a)$
while $b^{-1}a^{-1}d\notin {\mathcal C}(ba),$ since $bab^{-1}a^{-1}d=c d c^{-1}.$
\end{proof}

\begin{lem}\label{suff}
Let $x\in \Gamma_2$ and $u_1\dots u_{|x|}$ be a geodesic word representing $x$.
Assume the suffix $u_s=u_{|x|-2}u_{|x|-1} u_{|x|}\notin \mathcal{R}$.
Then $\mathcal{C}(x)=\mathcal{C}(u_s).$

The same conclusion holds if $u_s=u_{|x|-1} u_{|x|}.$
\end{lem}
\begin{proof} The proof is an application of Lemma \ref{lemmatutto}.

Let $u_p=u_{1}\dots u_{|x|-3},$ and consider a geodesic word
$y_{1}\dots y_{|x|-3}=u_p$. Note that $u_s\in\mathcal{C}(u_p)$ since
$u_1\dots u_{|x|}$ is a geodesic word.

If $y_{|x|-3}u_s\in \mathcal{R}$, then also
$u_s\in \mathcal{R}$, which is a contradiction. Hence
$y_{|x|-3}u_s\notin \mathcal{R},$ and, by Lemma \ref{lemmatutto}, we get
$\mathcal{C}(x)=\mathcal{C}(u_p u_s)=\mathcal{C}(u_s).$
\end{proof}

As a consequence of  Lemmas \ref{lemmatutto}, \ref{lemma4}, and \ref{uni},  we have
\begin{prop}\label{prop1}
All the possible cone types of $\Gamma_2$ (${\mC}(e)$ excluded) are those determined by geodesic words obtained as sub-words,
with length at most $4$, of any cyclic permutation of the relator $[a,b][c,d]$
or of its inverse $[d,c][b,a],$ adding up to $48$.
\end{prop}
\begin{proof}
The first sentence is a consequence of the previous Lemmas. The second one is established
by looking for words of length at most $4$  in any cyclic permutation of either $aba^{-1}b^{-1}cdc^{-1}d^{-1},$
or its inverse $dcd^{-1}c^{-1}bab^{-1}a^{-1},$ and adding up. Precisely we have:
\begin{itemize}
\item $8$ words of length $1$;
\item $2(7+1)$ words of length $2,$
($7$ in the  relator, $1$ in the following permutation, twice);
\item $2(6+1+1)$ words of length $3,$
($6$ in the  relator,  $1$ in the following $2$ permutations, twice);
\item $(5+1+1+1)$ words of length $4,$
($5$ in the relator, $1$ in the following $3$ permutation,
just once, since for the inverse you get the same words).
\end{itemize}

Adding up we get $6\times 8=48.$
\end{proof}
The above reasoning can be applied to a generic surface group
$\Gamma_g,$ leading to
\begin{thm}\label{main}
All the possible cone types of $\Gamma_g$ (${\mC}(e)$ excluded)
are those determined by geodesic words obtained as sub-words,
with length at most $2g$, of any cyclic permutation of either
$[a_1,b_1]\dots[a_g,b_g],$
or  its inverse $[b_g,a_g]\dots [b_1,a_1],$ adding up to $8g(2g-1)$.
\end{thm}
\begin{proof}
Looking for words of length at most $2g$ in
any cyclic permutation of either the relator
$[a_1,b_1]\!\dots[a_g,b_g]$
or  its inverse $[b_g,a_g]\!\dots \![b_1,a_1],$ we have:
\begin{itemize}
\item $4g$ words of length $1$;
\item $2((4g-1)+1)$ words of length $2,$
($4g-1$ in the  relator, $1$ in the following permutation, twice);
\item $2((4g-2)+1+1)$ words of length $3,$
($4g-2$ in the  relator,  $1$ in the following $2$ permutations, twice);
\item \dots \dots
\item $2((4g-(2g-2))+(2g-2))$ words of length $2g-1,$
($4g-(2g-2)$ in the  relator,  $1$ in the following $2g-1$ permutations, twice);
\item $((4g-(2g-1))+(2g-1))$ words of length $2g,$
($4g-(2g-1)$ in the  relator, $1$ in the following $2g-1$ permutation,
just once, since for the inverse you get the same words).
\end{itemize}

Adding up we get
\begin{eqnarray*}
& &4g(1+\underbrace{2+2+\dots +2}+1)=4g(2+2(2g-2))\\
& & \hspace{2cm}2g-2 \\
& & = 8g(1+2g-2)= 8g(2g-1)
\end{eqnarray*}
\end{proof}

\section{Cone type of successors}\label{successors}

We say, in short, that a cone type of an element of $\Gamma_2$ is a quadruple,
triple, double, single, if
it is one of the $48$ cone types defined by elements in $\mR$, of length, respectively
 $4, 3, 2, 1.$

It is convenient to  organize  the cone types in  singles, doubles, triples,
and quadruples. We follow the order shown in Table \ref{tab}.

\begin{table}[ht]
\caption{Cone types in $\Gamma_2$.}
\begin{tabular}{|l|l|l|l|l|l|l|l|} \hline
    & singles &  & doubles   &      & triples    &      & quadruples\\ \hline
$1$	& $b^{-1}$	& $9$  & $b^{-1}c$      & $25$ & $b^{-1}cd$ & $41$ & $b^{-1}cdc^{-1}$\\ \hline
$2$	& $a$	  & $10$ & $b^{-1}a^{-1}$ & $26$ & $b^{-1}a^{-1}d$ & $42$ &	$aba^{-1}b^{-1}$\\ \hline
$3$ & $d$	 & $11$ &	$ab$     & $27$ & $aba^{-1}$	& $43$ & $dc^{-1}d^{-1}a$\\ \hline 	
$4$ & $c^{-1}$ & $12$ &	$a b^{-1}$ &	$28$ & $ab^{-1}a^{-1}$ & $44$ & $c^{-1}d^{-1}ab$ \\ \hline
$5$ & $d^{-1}$ & $13$ &	$dc^{-1}$      &	$29$ & $dc^{-1}d^{-1}$ & $45$ & $d^{-1}aba^{-1}$\\ \hline
$6$ & $c$ & $14$ &	$dc$ &	$30$ & $dcd^{-1}$ & $46$ &	$cdc^{-1}d^{-1}$ \\ \hline
$7$ & $b$ & $15$ &	$c^{-1}d^{-1}$      &	$31$ & $c^{-1}d^{-1}a$	& $47$ & $b a^{-1}b^{-1}c$	\\ \hline
$8$ & $a^{-1}$ & $16$ &	$c^{-1}b$ &	$32$ & $c^{-1}ba$ & $48$ &$a^{-1}b^{-1}cd$ \\ \hline
	& 		         & $17$ &	$d^{-1}a$    &	$33$ & $d^{-1}ab$  & &\\ \hline
  &	             & $18$ &	$d^{-1}c^{-1}$&	$34$  &	$d^{-1}c^{-1}b$ & &\\ \hline
	&              & $19$	&$cd$& $35$	&  $cdc^{-1}$ & &\\ \hline
  &              & $20$	&	$cd^{-1} $    &	$36$ & $cd^{-1}c^{-1}$  & &\\ \hline
	&              & $21$	&$b a^{-1}$     &	$37$	& $b a^{-1}b^{-1}$& &	\\ \hline
  &	             & $22$	&	$ba$&	$38$  &	$bab^{-1}$& &\\ \hline
	&              & $23$	&$a^{-1}b^{-1}$     &	$39$	&  $a^{-1}b^{-1}c$& &\\ \hline
  &              & $24$ &	$a^{-1}d$  &	$40$ & $a^{-1}dc$& &\\ \hline
%
%
\end{tabular}
\label{tab}
\end{table}

We now provide cone types for each successors of the $48$ elements in
Table \ref{tab}. As it will be shown in Proposition \ref{ugua}, the list of cone
types of successors depends only on the cone type.
\begin{itemize}
\item We start with the generator $a$.
Successors of $a$ are
$$aa,\; ad,\; ac^{-1},\; ad^{-1},\;a c,\;a b,\; a b^{-1},$$
and we need to find the cone type of the first $5$ only, since they are not in $\mR$.

We note that for each $u=a,d,c^{-1},d^{-1},c$,
the geodesic word $au\notin\mR.$
 Hence by Lemma \ref{lemmatutto}
we get that  ${\mC} (au)= {\mC} (u).$
 The same argument works for any of
$$b,\; c,\; d,\; a^{-1},\; b^{-1},\;c^{-1},\; d^{-1},$$
and we observe that successors of {\it singles} are either singles or doubles.
\item
We consider $ab$.
Successors of $ab$ are
$$aba,\; abd,\; abc^{-1},\; abd^{-1},\;a bc,\;a bb,\; a ba^{-1},$$
and we have to find the cone type of the first $6$ only, since the last one
determines a cone type by itself.

Note that for each $u=d,c^{-1},d^{-1},c,b$,
both the geodesic words $abu,$ and $bu,$ are not in $\mR$.
 Hence by Lemma \ref{lemmatutto} applied twice,
we get that  ${\mC} (abu)= {\mC} (bu)= {\mC} (u).$

On the other hand, $aba\notin\mR,$  while $ba\in\mR$.
So we conclude that
${\mC} (aba)= {\mC} (ba).$
The same argument works for any double listed in Table \ref{tab},
and we observe that successors of {\it doubles} can be singles,  doubles
or triples.

\item
Consider now $aba^{-1}$.
Successors of $aba^{-1}$ are
$$ aba^{-1}d,\; aba^{-1}c^{-1},\; aba^{-1}d^{-1},\;a ba^{-1}c,
\;a ba^{-1}b,\; a ba^{-1}a^{-1}, \; aba^{-1}b^{-1}$$
and we have to find the cone type of the first $6$ only.

For each $u=c^{-1},d^{-1},c,b,a^{-1}$,
both  geodesic words $ba^{-1}u,$ and $a^{-1}u,$ are not in $\mR$. Hence by Lemma \ref{lemmatutto}
we get that
$${\mC} (aba^{-1}u)= {\mC} (ba^{-1}u)={\mC} (a^{-1}u)= {\mC} (u).$$

On the other hand, $ba^{-1}d\notin\mR,$
while $a^{-1}d\in\mR$. So we conclude that
$${\mC} (aba^{-1}d)= {\mC} (ba^{-1}d)={\mC} (a^{-1}d).$$
The same argument  works for any of the triples listed in Table \ref{tab},
and we observe that successors of {\it triples} are singles, doubles or
quadruples.

\item
Finally let us consider $aba^{-1}b^{-1}=dcd^{-1}c^{-1}$.
Successors of $aba^{-1}b^{-1}$ are
$$\begin{array}{lll}
  aba^{-1}b^{-1}a,\;&
aba^{-1}b^{-1}d,\;&
aba^{-1}b^{-1}c^{-1},\\
aba^{-1}b^{-1}d^{-1},\; &
\; a ba^{-1}b^{-1}a^{-1},&
\; aba^{-1}b^{-1}b^{-1}.
\end{array}$$

Note that for $u=d,c^{-1},a, b^{-1}$,
all geodesic words $ba^{-1}b^{-1}u,$ $a^{-1}b^{-1}u,$ and $b^{-1}u$
 are not in $\mR$. Hence by Lemma \ref{lemmatutto},
we get that  ${\mC} (aba^{-1}b^{-1}u)= {\mC} (u).$

On the other hand, $ba^{-1}b^{-1}a^{-1},$ and $a^{-1}b^{-1}a^{-1}$ are not
in $\mR$, while $b^{-1}a^{-1}$ is. So we conclude that
$${\mC} (ba^{-1}b^{-1}a^{-1})= {\mC} (a^{-1}b^{-1}a^{-1})={\mC} (b^{-1}a^{-1}).$$

Similarly $ba^{-1}b^{-1}d^{-1}=cd^{-1}c^{-1}d^{-1},$ and $d^{-1}c^{-1}d^{-1}$ are not in $\mR$, while $c^{-1}d^{-1}$ is. So we conclude that
$${\mC} (ba^{-1}b^{-1}d^{-1})= {\mC} (d^{-1}c^{-1}d^{-1})={\mC} (c^{-1}d^{-1}).$$
The same argument above works for every quadruple listed in Table \ref{tab},
and we observe that successors of {\it quadruple} are either singles or doubles.

\end{itemize}

\begin{lem}\label{geo}
If the cone type of $x\in \Gamma_2$ is ${\mathcal C}(z)$ where $z\in\mR$, then there exists a geodesic word
$u_1 u_2\dots u_{|x|-|z|+1}\dots u_{|x|}$ which represents $x$, ending with $z$ (i.e. $u_{|x|-|z|+1}\dots u_{|x|}=z$), and
such that $u_{|x|-|z|}z\notin \mathcal{R}.$
\end{lem}
\begin{proof}
We first show that there exists a geodesic word
$u_1\dots u_{|x|}$ which represents $x$ and ends with $z$.

For any  geodesic word $x_1\dots x_J$ which represents $x$ let us denote
by $s_{x,k}=x_{J-k+1}\dots x_J,$ $1\leq k\leq 4$ the longest
suffix  which belongs to $\mathcal{R}$.

Let $n=|z|,$  $1\leq n\leq 4,$ and, say, $z=z_1\dots z_n.$

Now assume, on the contrary,  that any geodesic word which represents $x$ does not end with $z$.
Then for any geodesic word $x_1\dots x_J$ such that
$x=x_1\dots x_J$ we have $x_{J-n+1}\dots x_J\neq z_1\dots z_n.$

If for some representation we have $k=4,$ then the cone type of $x$ is
${\mathcal C}(s_{x,4})$ by Lemma \ref{lemma4}, and by uniqueness of cone type we must have
$z=s_{x,4}$ which is a contradiction.

Hence for all other representations of $x$ we have $k<4$. Let's take one with the
longest suffix $s_{x,k}\in{\mathcal C}(x_1\dots x_{J-k})$
(it exists since otherwise we find a shorter geodesic representing $x$).

If for any geodesic word $y_1\dots y_{J-k}$ which represents $x_1\dots x_{J-k}$ we have that $y_{J-k}s_{x,k}\notin\mathcal{R}$ , then by  Lemma \ref{lemmatutto},
$${\mathcal C}(z)={\mathcal C}(x)={\mathcal C}(x_1\dots x_{J-k}s_{x,k})={\mathcal C}(s_{x,k}),$$
and again, by uniqueness, $z=s_{x,k}$, a contradiction.

So there exists a
geodesic word $y_1\dots y_{J-k}$ representing $x_1\dots x_{J-k}$ and $y_{J-k}s_{x,k}\in\mathcal{R}$.
Thus we have found a representation of $x=y_1\dots y_{J-k}s_{x,k},$
where the suffix $y_{J-k}s_{x,k}\in\mathcal{R}$ is longer then $s_{x,k}$,
yielding again a contradiction. This complete the first part of the proof.

Next, given a  geodesic word
$u_1 u_2\dots u_{|x|-|z|}z$ which represents $x$  we can exclude
$u_{|x|-|z|}z\in \mathcal{R}$ if $z$ is a quadruple, since the word is geodesic. In all other cases
we have, by Lemma \ref{lemma1},
$$\mathcal{C}(z)=\mathcal{C}(x)\subset\mathcal{C}(u_{|x|-|z|}z)\subset
\mathcal{C}(z),$$
and so, by Lemma \ref{uni}, we get  $u_{|x|-|z|}z\notin \mathcal{R}$.

\end{proof}

\begin{prop}\label{ugua}
If $x,z\in \Gamma_2,$ $z\in \mR,$ and $\mathcal{C}(x)=\mathcal{C}(z),$ then
the cone type of any of the successors of $x$ is the same as the cone type of
the successor of $z$ corresponding to the same generator.
\end{prop}
\begin{proof}
By Lemma \ref{geo} there exists a geodesic word which represents $x$, say
$u_1 u_2\dots u_{|x|-|z|+1}\dots u_{|x|},$  ending with $z$ (i.e. $u_{|x|-|z|+1}\dots u_{|x|}=z$), and
such that $u_{|x|-|z|}z\notin \mathcal{R}.$

Let us consider a successor $y$ of $x$. We have $y=xa$, with
$a\in A\cap \mathcal{C}(x)$ and
$d(e,xa)=d(e,x)+1$, yielding $a\in\mathcal{C}(z),$ and,
by definition of cone type, $d(e,za)=d(e,z)+d(e,a)=d(e,z)+1.$
Hence $za$ is a successor of $z$.

Also $a\in\mathcal{C}(x)$ implies that $u_1 u_2\dots u_{|x|-|z|}z a$ is a geodesic
word representing $xa$.

If $za\in\mathcal{R}$ is a quadruple, by Lemma \ref{lemma4} this yields
$\mathcal{C}(xa)=\mathcal{C}(za).$

If $za\in\mathcal{R}$ and  $|za|\leq 3,$ let $y=u_1 u_2\dots u_{|x|-|z|}$
and consider a geodesic word $v_1 v_2\dots v_{|x|-|z|}=y.$

Note that $za\in \mathcal{C}(y)$ so that $v_1 v_2\dots v_{|x|-|z|}z a$ is geodesic.

If $v_{|x|-|z|} z a\in\mathcal{R}$ then, also, $v_{|x|-|z|} z \in\mathcal{R}$
and, by Lemmas \ref{lemmatutto}, and \ref{lemma1},
$$\mathcal{C}(v_{|x|-|z|} z)\subset\mathcal{C}(z)=\mathcal{C}(x)\subset
\mathcal{C}(v_{|x|-|z|}z).$$

It follows $\mathcal{C}(v_{|x|-|z|} z)=\mathcal{C}(z),$ and by uniqueness of cone
types we have $v_{|x|-|z|} z=z$ and so the contradiction $v_{|x|-|z|}=e.$

Hence $v_{|x|-|z|} z a\notin\mathcal{R},$ and by Lemma \ref{lemmatutto}

$$\mathcal{C}(xa)=\mathcal{C}(yza)=\mathcal{C}(za).$$

So it remain to consider the case $za\notin\mathcal{R}.$

If $za\notin\mathcal{R}$ and $|za|\leq 3,$ we apply Lemma \ref{suff} to obtain
$\mathcal{C}(xa)=\mathcal{C}(za).$

If $za\notin\mathcal{R}$ and $|za|=4,$ then $|z|=3,$ and the cone type of
$z=z_1 z_2 z_3,$ is a triple.
Also, by Lemma \ref{lemmatutto} applied to
$y=z_1$,
$\mathcal{C}(z a)\mathcal{C}(z_2 z_3 a).$

Since the cone type of a successor of a triple can only be a
single, double, or quadruple,  it follows $z_2 z_3 a\not in \mathcal{R}.$

Therefore we have by Lemma \ref{suff},
$$\mathcal{C}(x a)=\mathcal{C}(z_2 z_3 a)=\mathcal{C}(z a).$$

The same reasoning apply for $za\notin\mathcal{R}$ and $|za|=5,$
recalling that quadruples  do not have triples and quadruples as successors.

If $za\notin\mathcal{R}$ and $|za|=5,$ then $|z|=4,$ and the cone type of
$z=z_1 z_2 z_3 z_4,$ is a quadruple.
Since the cone type of a successor of a quadruple can only be  either a
single or a  double,  it follows $z_2 z_3 z_4 a\not in \mathcal{R},$
(otherwise, by Lemma \ref{lemma4},
 $\mathcal{C}(z a)\mathcal{C}(z_1 z_3 z_4 a)$).

Hence, by Lemma \ref{lemmatutto} applied to
$y=z_1 z_2$,
$\mathcal{C}(z a)\mathcal{C}(z_3 z_4 a).$

Since the cone type of a successor of a quadruple can only be  either a
single or a  double,  it follows $z_3 z_4 a\not in \mathcal{R}.$

Therefore we have by Lemma \ref{suff},
$$\mathcal{C}(x a)=\mathcal{C}(z_3 z_4 a)=\mathcal{C}(z a),$$
and the proof is complete.
\end{proof}

\section{Matrix of cone types}\label{matrix}

Based on the discussion preceding Lemma \ref{geo}, and Proposition \ref{ugua} we can
now construct a matrix, indexed by cone types, in which any column gives the cone
type of any successor of the element whose cone type  indexes the column.
We work in the surface group of genus two, so we are speaking about a
 $48\times 48$ matrix. We think  it is better to provide the matrix as a block matrix.
 As we shall see, the matrix is sparse. It should be mentioned that any order of cone types
indexing its columns (and corresponding rows) gives a similar matrix, hence we follow
the order provided in Table \ref{tab}.

Next we list the $16$ blocks, $M_{i,j},$ $i,j=1,\dots,4,$ of the main matrix $M$.
Note that indexes $i,j$ refer to single, doubles, etc...
for example, indexes $M_{1,4}$ means  that rows are indexed by  singles, and
columns by quadruples. We set $0$ whenever $M_{i,j}$ is  the zero matrix, and $I$ for
the identity matrix.

\begin{equation}\label{emme}
M=\left(
\begin{array}{cccc}
M_{1,1}&M_{1,2}&M_{1,3}&M_{1,4}\\
M_{2,1}&M_{2,2}&M_{2,3}&M_{2,4}\\
0      &I       &0      &0\\
0      &0      &M_{4,3}&0
\end{array}
\right).
\end{equation}

The first block is an $8\times 8$ matrix

$$M_{1,1}=\left(
\begin{array}{cccccccc}
1 & 0& 1& 1& 1& 1& 0& 0\\
 1& 1 &1 &1 &0 &1 &0& 0\\
  1& 1& 1 &1 &0 &0 &1& 0\\
   1& 1& 0 &1 &0 &0 &1 &1\\
    1& 1& 0 &0 &1 &0 &1 &1\\
     0& 1& 0& 0& 1& 1& 1& 1 \\
    0 &0 &1 &0 &1 &1 &1 &1\\
     0& 0& 1& 1& 1& 1& 0& 1
    \end{array}\right).$$

In the same column it follows a $16\times 8$ matrix
$$M_{2,1}=\left(
\begin{array}{cccccccc}
1 &0& 0& 0& 0& 0& 0& 0\\
1& 0& 0& 0 &0 &0 &0 &0\\
 0 &1& 0& 0 &0& 0& 0& 0\\
 0 &1& 0& 0& 0& 0& 0& 0\\
 0 &0& 1& 0 &0& 0& 0& 0\\
 0 &0& 1& 0& 0& 0& 0& 0\\
 0 &0& 0& 1 &0& 0& 0& 0\\
 0 &0& 0& 1& 0& 0& 0& 0\\
 0 &0& 0& 0 &1& 0& 0& 0\\
 0 &0& 0& 0& 1& 0& 0& 0\\
 0 &0& 0& 0 &0& 1& 0& 0\\
 0 &0& 0& 0& 0& 1& 0& 0\\
 0 &0& 0& 0 &0& 0& 1& 0\\
 0 &0& 0& 0& 0& 0& 1& 0\\
 0 &0& 0& 0 &0& 0& 0& 1\\
 0 &0& 0& 0& 0& 0& 0& 1\\
\end{array}\right),$$

In the second column, first row, we have

$$M_{1,2}=
\left(
\begin{array}{cccccccccccccccc}
1&0&0&1&1&1&1&0&0&1&1&1&0&0&1&1\\
1&0&0&1&1&1&0&0&1&1&1&0&0&1&1&1\\
0&0&1&1&1&0&0&1&1&1&1&0&0&1&1&1\\
0&1&1&1&1&0&0&1&1&1&0&0&1&1&1&0\\
    0	&1	&1&	1	&0&	0	&1&	1	&1&	0	&0&	1	&1&	1	&1&	0\\
    1	&1	&1	&0	&0	&1	&1	&1	&1	&0	&0	&1	&1	&1	&0&	0\\
    1	&1	&1	&0	&0	&1	&1	&1	&0	&0	&1	&1	&1	&0	&0&	1\\
    1	&1  &0	&0  &1	&1  &1	&0  &0	&1&	1	&1&	1	&0  &0	&1
\end{array}\right),$$

and then, in the same column,
$$M_{2,2}=\left(
\begin{array}{cccccccccccccccc}
0&0&0&1&0&0&0&0&0&0&0&0&0&0&0&0\\
0&0&0&0&0&0&0&0&0&0&0&0&0&0&1&0\\
0&0&0&0&0&0&0&0&0&0&0&0&0&1&0&0\\
0&0&0&0&0&0&0&0&1&0&0&0&0&0&0&0\\
0&0&0&0&0&0&0&0&0&0&0&0&0&0&0&1\\
0&0&0&0&0&0&0&0&0&0&1&0&0&0&0&0\\
0&0&0&0&0&0&0&0&0&1&0&0&0&0&0&0\\
0&0&0&0&1&0&0&0&0&0&0&0&0&0&0&0\\
0&0&0&0&0&0&0&0&0&0&0&1&0&0&0&0\\
0&0&0&0&0&0&1&0&0&0&0&0&0&0&0&0\\
0&0&0&0&0&1&0&0&0&0&0&0&0&0&0&0\\
1&0&0&0&0&0&0&0&0&0&0&0&0&0&0&0\\
0&0&0&0&0&0&0&1&0&0&0&0&0&0&0&1\\
0&0&1&0&0&0&0&0&0&0&0&0&0&0&0&0\\
0&1&0&0&0&0&0&0&0&0&0&0&0&0&0&0\\
0&0&0&0&0&0&0&0&0&0&0&0&0&1&0&0
\end{array}\right).$$
In the third column, we have
$$M_{1,3}=\left(
\begin{array}{cccccccccccccccc}
1&1&0&0&1&1&0&0&0&0&1&1&1&1&1&1\\
1&1&0&0&0&0&1&1&0&0&1&1&1&1&1&1\\
1&1&0&0&0&0&1&1&1&1&1&1&1&1&0&0\\
0&0&1&1&0&0&1&1&1&1&1&1&1&1&0&0\\
0&0&1&1&1&1&1&1&1&1&0&0&1&1&0&0\\
0&0&1&1&1&1&1&1&1&1&0&0&0&0&1&1\\
1&1&1&1&1&1&0&0&1&1&0&0&0&0&1&1\\
1&1&1&1&1&1&0&0&0&0&1&1&0&0&1&1
\end{array}\right).$$

$$M_{2,3}=\left(
\begin{array}{cccccccccccccccc}
0&	0&	0&	0&	0&	0&	0&	0&	0&	0&	0&	0&	0&	1&	0&	0\\
0&	0&	0&	0&	0&	0&	0&	0&	0&	0&	0&	0&	1&	0&	0&	0\\
0&	0&	0&	0&	0&	0&	0&	1&	0&	0&	0&	0&	0&	0&	0&	0\\
0&	0&	0&	0&	0&	0&	1&	0&	0&	0&	0&	0&	0&	0&	0&	0\\
0&	1&	0&	0&	0&	0&	0&	0&	0&	0&	0&	0&	0&	0&	0&	0\\
1&	0&	0&	0&	0&	0&	0&	0&	0&	0&	0&	0&	0&	0&	0&	0\\
0&	0&	0&	0&	0&	0&	0&	0&	0&	0&	0&	1&	0&	0&	0&	0\\
0&	0&	0&	0&	0&	0&	0&	0&	0&	0&	1&	0&	0&	0&	0&	0\\
0&	0&	0&	0&	0&	1&	0&	0&	0&	0&	0&	0&	0&	0&	0&	0\\
0&	0&	0&	0&	1&	0&	0&	0&	0&	0&	0&	0&	0&	0&	0&	0\\
0&	0&	0&	0&	0&	0&	0&	0&	0&	0&	0&	0&	0&	0&	0&	1\\
0&	0&	0&	0&	0&	0&	0&	0&	0&	0&	0&	0&	0&	0&	1&	0\\
0&	0&	0&	0&	0&	0&	0&	0&	0&	1&	0&	0&	0&	0&	0&	0\\
0&	0&	0&	0&	0&	0&	0&	0&	1&	0&	0&	0&	0&	0&	0&	0\\
0&	0&	0&	1&	0&	0&	0&	0&	0&	0&	0&	0&	0&	0&	0&	0\\
0&	0&	1&	0&	0&	0&	0&	0&	0&	0&	0&	0&	0&	0&	0&	0
\end{array}\right).$$

$$M_{4,3}=\left(
\begin{array}{cccccccccccccccc}
1&	0&	0&	1&	0&	0&	0&	0&	0&	0&	0&	0&	0&	0&	0&	0\\
0&	0&	1&	0&	0&	1&	0&	0&	0&	0&	0&	0&	0&	0&	0&	0\\
0&	0&	0&	0&	1&	0&	0&	1&	0&	0&	0&	0&	0&	0&	0&	0\\
0&	0&	0&	0&	0&	0&	1&	0&	0&	1&	0&	0&	0&	0&	0&	0\\
0&	0&	0&	0&	0&	0&	0&	0&	1&	0&	0&	1&	0&	0&	0&	0\\
0&	0&	0&	0&	0&	0&	0&	0&	0&	0&	1&	0&	0&	1&	0&	0\\
0&	0&	0&	0&	0&	0&	0&	0&	0&	0&	0&	0&	1&	0&	0&	1\\
0&	1&	0&	0&	0&	0&	0&	0&	0&	0&	0&	0&	0&	0&	1&	0
\end{array}\right).$$

Finally in the last column we have
$$M_{1,4}=\left(
\begin{array}{cccccccc}
1&	1&	0&	0&	0&	0&	1&	1\\
1&	1&	1&	0&	0&	0&	0&	1\\
1&	1&	1&	1&	0&	0&	0&	0\\
0&	1&	1&	1&	1&	0&	0&	0\\
0&	0&	1&	1&	1&	1&	0&	0\\
0&	0&	0&	1&	1&	1&	1&	0\\
0&	0&	0&	0&	1&	1&	1&	1\\
1&	0&	0&	0&	0&	1&	1&	1
\end{array}\right).$$

$$M_{2,4}=\left(
\begin{array}{cccccccc}
0&	0&	1&	0&	0&	0&	0&	0\\
0&	1&	0&	0&	0&	0&	0&	0\\
0&	0&	0&	1&	0&	0&	0&	0\\
0&	0&	1&	0&	0&	0&	0&	0\\
1&	0&	0&	0&	0&	0&	0&	0\\
0&	0&	0&	0&	0&	0&	0&	1\\
0&	1&	0&	0&	0&	0&	0&	0\\
1&	0&	0&	0&	0&	0&	0&	0\\
0&	0&	0&	0&	0&	0&	1&	0\\
0&	0&	0&	0&	0&	1&	0&	0\\
0&	0&	0&	0&	0&	0&	0&	1\\
0&	0&	0&	0&	0&	0&	1&	0\\
0&	0&	0&	0&	1&	0&	0&	0\\
0&	0&	0&	1&	0&	0&	0&	0\\
0&	0&	0&	0&	0&	1&	0&	0\\
0&	0&	0&	0&	1&	0&	0&	0
\end{array}\right).$$

\begin{defn}\label{pri}
A square non-negative matrix $T$ is said to be primitive if there
exists a positive integer $k$ such that $T^k > 0$.
\end{defn}

\begin{prop}
The matrix $M$ is a primitive matrix.
\end{prop}
\begin{proof}
A direct computation shows that $M^5>0.$

More precisely,
all elements in the first $8$ rows of $M^2$ are strictly positive,
and the same for the first $16$ rows of $M^3$,
the first $32$ rows of $M^4$, and
all rows in  $M^5$.
\end{proof}

As a consequence, by Perron-Frobenius Theorem \cite[Theorem 1.1]{Se},
we obtain the following
\begin{prop}
There exists an eigenvalue $r$ of $M$ such that:
\begin{itemize}
\item[1)] $r$ is real, and $r >0$;
\item[2)]  $r$ is associated to strictly positive left and right eigenvectors;
\item[3)] $r > |\lambda|$  for any eigenvalue $\lambda\neq r$;
\item[4)] The eigenvectors associated with $r$ are unique up to constant multiples;
\item[5)] $r$ is a simple root of the characteristic equation of $T.$
\end{itemize}
\end{prop}
\section{Elementary multiplicative functions}\label{multi}

In this section we recall the definition of vector-valued elementary multiplicative functions on $\G,$
by Kuhn and Steger, and we show how they can be easily computed by means of the matrix
of cone types.

To set the vectorial context we need to define maps between finite
dimensional vector spaces indexed by cone types. In this section we use
letters as $a,b,\dots$ to denote arbitrary generators, not to be confused with
the set of generators provided in \eqref{gamma2}.

We consider triples $(a,c',c),$ where $c,c'$ are cone types, and $a\in A.$
We call a triple {\em admissible} if, given
$c=\mC(z)$ for some $z\in\G,$ we have $a\in A\cap {\mC}(z),$ (so that  if $z$ is
represented by a geodesic word $u$, then  $ua$ is also a geodesic) and
$c'=\mC(za),$ i.e. $c'$ is the cone-type of the $a$-successor of $z.$

 A matrix system (system in short) $(V_c, H_{a,c',c})$ consists of
finite dimensional complex vector spaces $V_c,$ for each cone-type $c$, and linear maps
$H_{a,c',c}:V_c\rightarrow V_{c'}$ for each admissible triple $(a,c',c).$
For non-admissible triples $(a,c',c)$ we set $H_{a,c',c}=0$.

\begin{defn}\label{mult}
For $x,y\in G$, consider the cone $\mC(x,y)$ with cone-type
$c=\mC(x^{-1}y)$
, $v_c\in V_c.$

The elementary multiplicative function $\mu[{\mC}(x,y),v_c]$ is defined as
$$
\mu[{\mC}(x,y),v_c](z)=\left\{\begin{array}{ll}
0,&z\notin{\mC}(x,y),\\ \\
v_c,& z=y,\\ \\
\!\!\displaystyle\sum_{\atopn{\atopn{a\in A}{ya\in{\mathcal C}(x,y)}}
{c'={\mathcal C}(x^{-1}ya)}} 
{\mu[{\mC}(x,ya), H_{a,c',c}(v_c)](z)},& z\neq y.
\end{array}\right.
$$
\end{defn}

Note that the action of $\Gamma_2$ on cones by translation (see Definition \ref{tk}),
$\gamma{\mC}(x,y)={\mC}(\gamma x,\gamma y)$, $\gamma\in\Gamma_2$, implies
\begin{equation}\label{action}
\mu[{\mC}(x,y),v_c](\gamma^{-1}z)=
\mu[{\mC}(\gamma x,\gamma y),v_c](z).
\end{equation}

The recursive definition yields an equivalent definition of elementary multiplicative function
 in terms of geodesics between two vertices.

Note that any geodesic between two vertices has the same length.
Also, since $\G$ is hyperbolic, the number
of geodesics between two fixed vertices is finite.

We recall that elements of $\mR$ are geodesic  sub-word of a
cyclic permutation of the fundamental relation $[a,b][c,d]=e$
or of its inverse $[d,c][b,a]=e.$

Consider any  geodesic word representing $y\in G$,
say $y=w_1\dots w_n,$  $w_i\in A,$ and the set of all geodesic quadruples
in $w_1w_2\dots w_n$ which belong to $\mR.$

\begin{defn}
For $y\in G$, $\mR_y$ is defined as the set of all geodesic quadruples, in any
geodesic word representing $y,$  which belong to $\mR$, i.e.
\begin{eqnarray*}
\mR_y&=&\left\{w_{i}w_{i+1}w_{i+2}w_{i+3}\in \mR,\right.\\
& &\left.\textrm{for all geodesic words}\, w_1w_2\dots w_n=y,\, w_i\in A\right\}.
\end{eqnarray*}
If $q\in \mR_y$,  the quadruple  $q'$ is called the \textit{twin} of
$q$
if $q=q'$ represents the same group element, but $q\neq q'$ as a geodesic path in
the Caley graph ($q'$ always exists).
\end{defn}

Note that if $q\in \mR_y,$ and $q'$ is the twin of $q,$ then also $q'\in R_y.$
Also $R_y=\emptyset$ means that, if $y\in\mC(e,a),$ for a given $a\in A$, there is
only one geodesic from $e$ to $y$ passing through $a$.

\begin{rem}\label{more}
If a given geodesic word  contains more then one quadruple
in $\mR_y,$ then they can have  at most one element in common, since any two
octagons have at most one edge in common.
One such example is:
$$y=(aba^{-1}b^{-1})a^{-1}d c=(d c d^{-1} c^{-1})a^{-1}d c=
aba^{-1}(a^{-1}b^{-1}c d).$$
Also note that the replacement of a quadruple $q$ by its twin $q'=q$ can change
the number of quadruples in $R_y.$

  In the  example above, $aba^{-1}b^{-1}a^{-1}d c$ contains $2$ quadruples:
$q_1=aba^{-1}b^{-1}$ and $q_2=b^{-1}a^{-1}d c$, but
if we replace $q_1$ with $q'_1=q_1$ then the number of quadruples drops to $1$.
This reflects the number of different geodesic paths from $e$ to $y$.

Hence it is necessary to introduce a  notation which takes into account
the several occurrences of quadruples.
\end{rem}

\begin{notat}\label{ord}
The notation follows a hierarchical
dyadic approach.

Since both the number of geodesic words representing $y$ and the number of
generators is finite, the {\em first quadruple} in $R_y$  can be defined as follows.
First let us fix an order on generators and on its inverses.

Let $i=1,\dots, n$, be the smallest index such that
$y=w_1w_2\dots w_n,$  $w_{i}w_{i+1}w_{i+2}w_{i+3}\in R_y,$   and $w_i$
is the smallest in the given order on generators.

We set $w_{0,j}=w_{j},$ for $j=i,\dots,i+3$ and we define the {\em first quadruple} in $R_y$ as
$q_0=w_{0,i}w_{0,i+1}w_{0,i+2}w_{0,i+3};$
its twin in $R_y$ is denoted  by $q_1=w_{1,i}w_{1,i+1}w_{1,i+2}w_{1,i+3}$.
Hence $q_0=q_1$ as a geodesic word, but $q_0\neq q_1$ as a geodesic path.

If we set $\delta_i=0,1,$  at the
$n$th stage $q_{(\delta_1,\dots,\delta_{n-1},0)}$ denotes the next quadruple in $R_y$ after
$q_{(\delta_1,\dots,\delta_{n-1})}$, while
$q_{(\delta_1,\dots,\delta_{n-1},1)}$ denotes the twin of $q_{(\delta_1,\dots,\delta_{n-1},0)}$.

In the example above, ordering as in Table \ref{tab}, $$b^{-1}<a<d<c^{-1}<d^{-1}<c<b<a^{-1},$$
we have
$$\begin{array}{lcl}
y&=&(aba^{-1}b^{-1})a^{-1}d c
=(d c d^{-1} c^{-1})a^{-1}d c
= aba^{-1}(a^{-1}b^{-1}c d).\\ \\
q_0&=&aba^{-1}b^{-1},\quad q_1=d c d^{-1} c^{-1}\\
q_{(0,0)}&=&b^{-1}a^{-1}d c,\quad
q_{(0,1)}=a^{-1}b^{-1}c d.
\end{array}$$

Finally, we shall use the notation $\delta=(\delta_1,\dots,\delta_{n}),$
and $q_{\delta}$ for short; if $\delta=(\delta_1,\dots,\delta_{n-1},0),$
and
$$q_{\delta}=w_{\delta_1,\dots,\delta_{n-1},0,i}
w_{\delta_1,\dots,\delta_{n-1},0,i+1}w_{\delta_1,\dots,\delta_{n-1},0,i+2}
w_{\delta_1,\dots,\delta_{n-1},0,i+3},$$
then we shall indicate its twin by setting ${\delta}'=(\delta_1,\dots,\delta_{n-1},1),$
and
$$q_{{\delta}'}=w_{\delta_1,\dots,\delta_{n-1},1,i}\dots w_{\delta_1,\dots,\delta_{n-1},1,i+3}.$$

If there is no confusion, we shall omit $\delta$ at all.

\end{notat}

The following lemmas will be crucial for an alternative expression of elementary
multiplicative functions.

\begin{lem}\label{app1}
Let $x,y\in G$ be  such that $R_y\neq\emptyset,$
and $y\in\mC(x).$ 

Let $i=1,\dots, n$, be the smallest index such that
$y=w_1w_2\dots w_n,$ and
$w_{i}w_{i+1}w_{i+2}w_{i+3}\in R_y.$
Denote it by $w_{0,i}w_{0,i+1}w_{0,i+2}w_{0,i+3},$
and its twin by $w_{1,i}w_{1,i+1}w_{1,i+2}w_{1,i+3}.$

Let $q_0$ be the first quadruple in $R_y,$ say $y=w_1w_2\dots w_n$ and
$q_0=w_{0,i}w_{0,i+1}w_{0,i+2}w_{0,i+3}.$

Let $a\in A\cap \mC(x).$ 
\begin{itemize}
\item[] If $i=1,$ then
$$xy\in{\mC}(e,xa)\Leftrightarrow \text{either} \;w_{0,1}=w_1=a, \text{or}\;
w_{1,1}=a;$$
\item[] If $ i>1,$ then
$$xy\in{\mC}(e,xa)\Leftrightarrow w_1=a.$$
\end{itemize}
\end{lem}
\begin{proof}
First note that
$xy\in{\mC}(e,x a)$ means
$$d(e,x y)=d(e,x a)+d(x a,x y)=d(e,x a)+d(a,w_1w_2\dots w_n).$$
Since $y\in\mC(x),$ and $a\in \mC(x),$ the latter is equivalent to
$$ |x|+n
=|x|+1+d(a,w_1w_2\dots w_n).$$
Therefore $xy\in{\mC}(e,xa)$ is equivalent to
\begin{equation}\label{erre}
n-1=d(a,w_1w_2\dots w_n)=d(e,a^{-1}w_1w_2\dots w_n).
\end{equation}
If $ i=1,$ and $a^{-1}w_1\neq e,$ then \eqref{erre} implies
$a^{-1}w_1w_2w_3w_4\in R$ and it equals a geodesic word of length $3$.
The latter implies  also $w_1w_2w_3w_4\in R.$ Being
$$w_{1}w_2w_{3}w_{4}=w_{0,1}w_{0,2}w_{0,3}w_{0,4}=w_{1,1}w_{1,2}w_{1,3}w_{1,4},$$
then
$$a^{-1}w_{1}w_2w_{3}w_{4}=a^{-1}w_{1,1}w_{1,2}w_{1,3}w_{1,4},$$
and both are equal to a geodesic word of length $3,$
hence necessarily $a^{-1}w_{1,1}=e.$

If $ i>1,$ then $w_{1}w_2w_{3}w_{4}\notin R.$ Hence \eqref{erre} implies $a^{-1}w_1=e$.

The reversed implications are immediate.
\end{proof}

A similar reasoning leads to the following lemma.
\begin{lem}\label{app2}
Let $x,y\in G,$  such that $R_y=\emptyset,$ and $y\in\mC(x).$

Let $a\in A\cap\mC(x).$  Then, if  $w_1w_2\dots w_n$ is the (only)  geodesic word representing $y,$ we have
 $$xy\in{\mC}(e,xa)\Leftrightarrow w_1=a.$$
\end{lem}

In the following, the
value of an elementary multiplicative function $\mu[{\mC}(e,b), v]$
in $z\in\mC(e,b)$ is obtained as a (finite) sum over all
geodesic paths from $b$ to $z$.

\begin{prop}\label{multgeo}
Let $b\in A,$ $c_0=\mC(b),$ and $v_{c_0}\in V_{c_0}.$ Then, for $z\in{\mC}(e,b),$ $z\neq b,$ we have
$$
\mu[{\mC}(e,b), v_{c_0}](z)=\sum_{\atopn{\text{geodesic words}\,w_1w_2\dots w_n}{b^{-1}z=w_1w_2\dots w_n}}
{\left[\prod_{j=1}^n{H_{w_j,c_j,c_{j-1}}}\right]}(v_{c_0}),
$$
with notation $c_0=\mC(b),$ $c_j=\mC(bw_1\dots w_j),$ and
$$\prod_{j=1}^n{H_{w_j,c_j,c_{j-1}}}= H_{w_n,c_n,c_{n-1}} H_{w_{n-1},c_{n-1},c_{n-2}}\dots H_{w_1,c_1,c_0}.$$
\end{prop}

\begin{proof}
Since $z\in{\mC}(e,b),$ $z\neq b,$  then $d(e,z)=d(e,b)+d(b,z)$ and
we can write $b^{-1}z=w_1w_2\dots w_n,$ where
$w_1w_2\dots w_n$ is a geodesic word, $n=|z|-1.$ Set $y=b^{-1}z.$

We have
either $R_y=\emptyset,$ or $R_y\neq \emptyset.$

In the first case, $w_1w_2\dots w_n$ is the only  geodesic word representing
$y$, hence by Definition \ref{mult} and  Lemma \ref{app2} applied to $x=b,$
\begin{eqnarray*}
\mu[{\mC}(e,b), v_{c_0}](z)\!\!\!&=&\!\!\!\!
\sum_{\atopn{\atopn{b'\in A}{bb'\in{\mathcal C}(e,b)}}{c'={\mathcal C}(bb')}}
{\mu[{\mC}(e,bb'), H_{b',c',c_0}(v_{c_0})](z)}\\ \\
&=&\mu[{\mC}(e,b w_1), H_{w_1,c_1,c_0}(v_{c_0})](b w_1\dots w_n)
\end{eqnarray*}
where $c_0=\mC(b),$ $c_1=\mC(b w_1),$ since all other instances are null.
A repeated application of Lemma \ref{app2} yields
\begin{eqnarray*}
& &\mu[{\mC}(e,b), v_{c_0}](b w_1\dots w_n)\\
&=&
\mu[{\mC}(e,b w_1\dots w_n), H_{w_n,c_n,c_{n-1}}\dots H_{w_1,c_1,c_0}(v_{c_0})]
(bw_1\dots w_n)\\
&=&H_{w_n,c_n,c_{n-1}}\dots H_{w_1,c_1,c_0}(v_{c_0}),
\end{eqnarray*}
as desired.

If $R_y\neq \emptyset,$ instead, we consider the first  quadruple
$q_0\in R_y$, and its twin $q_1$,
$$q_0=w_{0,i_0}\dots w_{0,i_0+3}= w_{1,i_1}\dots w_{1,i_1+3}=q_1.$$
We consider the quadruple next to $q_0$, if any, and its twin, say
$$q_{(0,0)}=w_{(0,0),i_{(0,0)}}\dots w_{(0,0),i_{(0,0)}+3}=
w_{(0,1),i_{(0,1)}}\dots w_{(0,1),i_{(0,1)}+3}
=q_{(0,1)},$$
where ${i_0+2}< {i_{(0,0)}}$ and ${i_0+3}\leq {i_{(0,0)}}.$

Similarly for $q_1$, if any, say
$$q_{(1,0)}=w_{(1,0),i_{(1,0)}}\dots w_{(1,0),i_{(1,0)}+3}=
w_{(1,1),i_{(1,1)}}\dots w_{(1,1),i_{(1,1)}+3}=q_{(1,1)},$$
with $i_1+3\leq i_{(1,0)},$  and so on so forth...

Consider $q_0.$
Lemma \ref{app1} implies
\begin{eqnarray*}
& &\mu[{\mC}(e,b), v_{c_0}](z)=
\sum_{\atopn{\atopn{b'\in A}{bb'\in{\mathcal C}(e,b)}}{c'={\mathcal C}(bb')}}
{\mu[{\mC}(e,bb'), H_{b',c',c_0}(v_{c_0})](z)}\\ \\
\!\!\!\!\!\!&=&\!\!\!\!
\left\{\!\!
\begin{array}{ll}\!\!
{\begin{array}{l}
\mu[{\mC}(e,bw_{0,1}), H_{w_{0,1},c_{0,1},c_0}(v_{c_0})](bw_1\dots w_4\dots w_n)\\ \\
+ \mu[{\mC}(e,bw_{1,1}), H_{w_{1,1},c_{1,1},c_0}(v_{c_0})](bw_{1,1}\dots w_{1,4}\dots w_n),
\end{array}}
&\!\!\!\text{if}\; i_0=1,\\ \\ \\
\mu[{\mC}(e,bw_{1}), H_{w_{1},c_{1},c_0}(v_{c_0})](bw_1\dots w_4\dots w_n),&\!\!\!\text{if}\; i_0>1.
\end{array}
\right.
\end{eqnarray*}
where, in the $i_0=1$ case,  $w_{0,1}=w_1,$ $c_1=c_{0,1}=\mC(bw_{0,1}),$ and $c_{1,1}=\mC(bw_{1,1}).$

 Since the subsequent quadruple, if any,
 has  in common with the previous one
at most one element (either $w_4$ or $w_{1,4}$), see Remark \ref{more},
 by Lemma~\ref{app1}, with obvious meaning of  symbols,  we get
in the  $i_0=1$ case (if there is no subsequent
quadruple, we set $i_{(\delta_j,0)}=n+1$)
\begin{eqnarray*}
& &\!\!\!\mu[{\mC}(e,b), v_{c_0}](z)\\ \\
&\!\!\!\!\!\!=&\!\!\!
\mu[{\mC}(e,bw_{0,1}w_{0,2}w_{0,3}), H_{w_{0,3},c_{0,3},c_{0,2}}\dots H_{w_{0,1},c_{0,1},c_0}(v_{c_0})](bw_1\dots 
w_n)\\ \\
&\!\!\!\!\!\! +&\!\!\! \mu[{\mC}(e,bw_{1,1}w_{1,2}w_{1,3}), H_{w_{1,3},c_{1,3},c_{1,2}}\dots
H_{w_{1,1},c_{1,1},c_0}(v_{c_0})](bw_{1,1}\dots 
w_n)\\ \\
&\!\!\!\!\!\! =&\dots\\
&\!\!\!\!\!\!=&\!\!\!
\mu[{\mC}(e,bw_{0,1}
\dots w_{i_{(0,0)}-1}),
H_{w_{i_{(0,0)}-1},c_{i_{(0,0)}-1},c_{i_{(0,0)}-2}}\dots\\
& &\hspace{.5cm} \dots
H_{w_{0,3},c_{0,3},c_{0,2}}\dots
 H_{w_{0,1},c_{0,1},c_0}(v_{c_0})](bw_1
\dots w_{i_{(0,0)}-1}\dots w_n)\\ \\
&\!\!\!\!\!\!+ &\!\!\!\!
\mu[{\mC}(e,bw_{1,1}
\dots w_{i_{(1,0)}-1}), H_{w_{i_{(1,0)}-1},c_{i_{(1,0)}-1},c_{i_{(1,0)}-2}}\dots\\
& &\hspace{.5cm}\dots H_{w_{1,3},c_{1,3},c_{1,2}}\dots
H_{w_{1,1},c_{1,1},c_0}(v_{c_0})](bw_{1,1}\dots w_{i_{(1,0)}-1}\dots  w_n),
\end{eqnarray*}

while, if $i_0>1,$ we get
\begin{eqnarray*}
& &\!\!\mu[{\mC}(e,b), v_{c_0}](z)\\
&=&\!\!\mu[{\mC}(e,bw_{1}\dots w_{i_{0}-1}), H_{w_{i_{0}-1},c_{i_{0}-1},c_{i_{0}-2}}
\dots H_{w_{1},c_{1},c_0}(v_{c_0})](bw_1\dots w_n).
\end{eqnarray*}

We apply Lemma \ref{app1} recursively to any summand so generated, for any
subsequent quadruple and  corresponding twin
which contribute to a different way of writing $y$ as a geodesic word.

After a finite number of steps,we get
\begin{eqnarray*}
& &\mu[{\mC}(e,b), v_{c_0}](z)\\
&=&
\sum_{\atopn{\text{geodesic words}\, bw_1w_2\dots w_n}{\text{such that}\;z=bw_1w_2\dots w_n}}
\mu[{\mC}(e,bw_{1}\dots w_{n}), H_{w_{n},c_{n},c_{n-1}}\dots\\
& & \hspace{6cm}  \dots H_{w_{1},c_{1},c_0}(v_{c_0})](bw_1\dots w_n)\\ \\
&=&
\sum_{\atopn{\text{geodesic words}\, w_1w_2\dots w_n}{\text{such that}\;b^{-1}z=w_1w_2\dots w_n}}
{\left[\prod_{j=1}^n{H_{w_j,c_j,c_{j-1}}}\right]}(v_{c_0}).
\end{eqnarray*}
\end{proof}

We  provide now a realization of elementary multiplicative functions in terms
of the matrix of cone types $M$.
The key point is that any admissible triple $(a,c',c)$ is independent from $a.$
\begin{prop}\label{piatta}
Let $x,y\in\G$ such that $\mC(x)=\mC(y)$. Let $a\in A\cap \mC(x)$ and $b\in A\cap \mC(y).$
If $\mC(xa)=\mC(yb)$ then $a=b.$

As a consequence, if triples $(a,c',c),$ $(b,c',c)$ are admissible, then $a=b$
(and we can simply write $(c',c)$).
\end{prop}
\begin{proof}
There exists $z\in\mR$ such that $\mC(x)=\mC(y)=\mC(z).$
By Proposition \ref{ugua}, we get
$$\mC(za)=\mC(xa)=\mC(yb)=\mC(zb),$$
where  $a,b\in A\cap \mC(z).$ By results of Section \ref{successors} we know that
different successors of $z\in\mR$ have different cone types, yielding $za=zb$ and so $a=b.$
\end{proof}

Let us consider the scalar case, first, where any matrix system is as follows:
$V_c=\C,$ for each cone type $c$, and, for any admissible triple
$({a,c',c}),$
 $H_{a,c',c}$  is multiplication by a non-zero complex number,  while $H_{a,c',c}=0$ otherwise.
Let $M=(m_{c',c})$ be the matrix \eqref{emme}, indexed by cone types.

We note that, for any  couple of cone types ${c',c}$ and for any $a\in A,$
 $m_{c',c}H_{a,c',c}=H_{a,c',c},$ and the latter is non-zero only for
admissible triples. Hence, by Proposition \ref{piatta}, we can set $H_{c',c}=m_{c',c}H_{a,c',c}$ and
consider a new matrix
$N=(H_{c',c}),$ with non-zero entries in the same positions as  $M$.

Also, for any cone type $c,$ let $V_c=(0\dots 1\dots 0)$ be the vector with $1$ at the $c$-position, and
$E_{c}=V_c^{\top}V_c $ be the $48\times 48$ matrix whose entries are all null except
the $(c,c)$ diagonal element (equal to $1$).

\begin{coro}[Scalar case]\label{multmatrix}
Let $b\in A,$ $c_0=\mC(b),$ and $v\in \C.$ Then,
for $z\in{\mC}(e,b),$ $z\neq b,$ we have
\begin{eqnarray*}
& &\mu[{\mC}(e,b),v](z)=
\sum_{\atopn{\text{geodesic words}\,w_1w_2\dots w_n}{b^{-1}z=w_1w_2\dots w_n}}
{\left[\prod_{j=1}^n{H_{c_j,c_{j-1}}}\right]} v\nonumber \\ \nonumber \\ 
& & =
V_{c_n}N\left[\sum_{\atopn{\text{geodesic words}\,w_1w_2\dots w_n}{b^{-1}z=w_1w_2\dots w_n}}
{E_{c_{n-1}} N E_{c_{n-2}}\dots N E_{c_1}}\right] N V_{c_0}^{\top}v,
\end{eqnarray*}
with notation $c_0=\mC(b),$ $c_j=\mC(bw_1\dots w_j),$ $c_n=\mC(z).$
\end{coro}
\begin{proof}
It is easy to see that, for any triple $(a,c',c),$
$$V_{c'}N V_{c}^{\top}=(0\dots 1\dots 0) N \left(\begin{array}{c}
0\\
\vdots\\
1\\
\vdots\\
0
\end{array}
\right)= H_{c',c},$$
and that the latter is $H_{a,c',c}$ if the triple is admissible, and $0$ otherwise.
\end{proof}

Also, by \eqref{action} we get a similar result for cones like ${\mC}(b,e).$
\begin{prop}\label{multgeo2}
Let $b\in A,$ $c_0=\mC(b^{-1}),$ and $v_{c_0}\in V_{c_0}.$ Then, for $z\in{\mC}(b,e),$ $z\neq e,$ we have
$$
\mu[{\mC}(b,e), v_{c_0}](z)=\sum_{\atopn{\text{geodesic words}\,w_1w_2\dots w_n}{z=w_1w_2\dots w_n}}
{\left[\prod_{j=1}^n{H_{w_j,c_j,c_{j-1}}}\right]}(v_{c_0}),
$$
where $c_0=\mC(b^{-1}),$ $c_j=\mC(b^{-1} w_1\dots w_j),$ and
$$\prod_{j=1}^n{H_{w_j,c_j,c_{j-1}}}= H_{w_n,c_n,c_{n-1}} H_{w_{n-1},c_{n-1},c_{n-2}}\dots H_{w_1,c_1,c_0}.$$

In the scalar case, with notation as in Corollary \ref{multmatrix}
\begin{eqnarray*}
& &\mu[{\mC}(b,e),v](z)=
\sum_{\atopn{\text{geodesic words}\,w_1w_2\dots w_n}{z=w_1w_2\dots w_n}}
{\left[\prod_{j=1}^n{H_{c_j,c_{j-1}}}\right]v}\nonumber \\ \nonumber \\ 
& & =
V_{c_n}N\left[\sum_{\atopn{\text{geodesic words}\,w_1w_2\dots w_n}
{z=w_1w_2\dots w_n}}
{E_{c_{n-1}} N E_{c_{n-2}}\dots N E_{c_1}}\right] N V_{c_0}^{\top}v,
\end{eqnarray*}
with notation $c_0=\mC(b^{-1}),$ $c_j=\mC(b^{-1}w_1\dots w_j),$ $c_n=\mC(b^{-1}z).$
\end{prop}
\begin{proof}
Since  $b^{-1}z\neq b^{-1},$ and $b^{-1}z\in\mC(e,b^{-1}),$  by \eqref{action} and
Proposition \ref{multgeo}
\begin{eqnarray*}
& &\mu[{\mC}(b,e), v_{c_0}](z)=\mu[{\mC}(e,b^{-1}), v_{c_0}](b^{-1}z)\\ \\
& &=
\sum_{\atopn{\text{geodesic words}\, w_1w_2\dots w_n}{z=w_1w_2\dots w_n}}
{\left[\prod_{j=1}^n{H_{w_j,c_j,c_{j-1}}}\right]}(v_{c_0}),
\end{eqnarray*}
where $c_0=\mC(b^{-1}),$ $c_j=\mC(b^{-1}w_1\dots w_j),$
$c_n=\mC(b^{-1}z)$ and
$$\prod_{j=1}^n{H_{w_j,c_j,c_{j-1}}}= H_{w_n,c_n,c_{n-1}} H_{w_{n-1},c_{n-1},c_{n-2}}\dots H_{w_1,c_1,c_0}.$$
\end{proof}

The vector case follows naturally. Let
$d_c$ be the dimension of the (finite dimension) vector space  $V_{c}.$
 $H_{a,c',c}$ can be identified with a $d_{c'}\times d_c$ matrix, and $H_{a,c',c}=0$ for
 non-admissible triples. Let $d=\sum_{c}d_{c}.$

As before, for any  couple of cone types ${c',c}$ and for any $a\in A,$
 multiplication by the scalar $m_{c',c}$ yields
  $m_{c',c}H_{a,c',c}=H_{a,c',c},$ and the latter is zero for
non-admissible triples. Hence, by Proposition \ref{piatta}, we can set $H_{c',c}=m_{c',c}H_{a,c',c}$ and
 define a (block) $d\times d$ matrix
$\mathcal{N}=(H_{c',c}).$

Also, let $\mathcal{V}_{c}=(0\dots I\dots 0)$ be the block
matrix with the identity matrix $I$ at the $d_c$-position, and
$\mathcal{E}_{c}=\mathcal{V}_c^{\top}\mathcal{V}_c $ be the $d\times d$ block
matrix whose entries are all null except
the $(d_c,d_c)$ diagonal element (equal to $I$).

\begin{coro}[Vector case]\label{multmatrixv}
Let $b\in A,$ $c_0=\mC(b),$ and $v\in V_{c_0}.$ Then,
for $z\in{\mC}(e,b),$ $z\neq b,$ we have
$$
\mu[{\mC}(e,b),v](z)
 =\!
\mathcal{V}_{c_n}\mathcal{N}\left[\sum_{\atopn{\text{geodesic words}\,w_1w_2\dots w_n}{b^{-1}z=w_1w_2\dots w_n}}
{\!\!\mathcal{E}_{c_{n-1}} \mathcal{N} \mathcal{E}_{c_{n-2}}\dots \mathcal{N} \mathcal{E}_{c_1}}\right] \mathcal{N} \mathcal{V}_{c_0}^{\top}v,
$$
with notation $c_0=\mC(b),$ $c_j=\mC(bw_1\dots w_j),$ $c_n=\mC(z).$
\end{coro}

Proposition \ref{multgeo2} extends similarly.

\section*{Acknowledgments}
The author would like to thank M. Gabriella Kuhn and Tim Steger  for fruitful and valuable discussions.

This work was supported by the Italian Research Project ``Progetto di Rilevante Interesse
Nazionale"  (PRIN) 2015: ``Real and complex manifolds: geometry, topology,
and harmonic analysis".

\end{document}